\documentclass[11pt]{article}

\usepackage{amssymb}
\usepackage{graphicx}
\usepackage{bbm}
\usepackage{mathrsfs}
\usepackage{amsmath}
\usepackage{amsthm}
\usepackage[center]{caption}
\usepackage{enumerate}
\usepackage{verbatim}
\usepackage{authblk}
\usepackage[hmargin=2.2cm, vmargin=2.8cm]{geometry}

\theoremstyle{plain}
\newtheorem{thm}{Theorem}[]

\newtheorem{lem}[thm]{Lemma}
\newtheorem{prop}[thm]{Proposition}

\theoremstyle{definition}

\newtheorem*{rmk}{Remarks}

\renewcommand{\P}{\mathbb{P}}
\newcommand{\E}{\mathbb{E}}
\newcommand{\R}{\mathbb{R}}
\newcommand{\F}{\mathcal{F}}
\newcommand{\G}{\mathcal{G}}
\newcommand{\Pt}{\tilde{\mathbb{P}}}

\newcommand{\hs}{\hspace{2mm}}
\newcommand{\hsl}{\hspace{1mm}}
\newcommand{\ind}{\mathbbm{1}}

\newcommand{\eps}{\varepsilon}
\newcommand{\bp}{\begin{proof}}
\newcommand{\ep}{\end{proof}}
\newcommand{\ec}{\mathcal{E}}
\newcommand{\smfr}[2]{\hbox{$\frac{#1}{#2}$}}
\def\bal#1\eal{\begin{align*}#1\end{align*}}

\title{Fine asymptotics for the consistent maximal displacement of branching Brownian motion}

\author{Matthew I.~Roberts\footnote{Department of Mathematical Sciences, University of Bath, Bath, BA2 7AY, UK. Email: \texttt{mattiroberts@gmail.com}}}

\begin{document}

\maketitle

\abstract{It is well-known that the maximal particle in a branching Brownian motion sits near $\sqrt2 t - \frac{3}{2\sqrt2}\log t$ at time $t$. One may then ask about the paths of particles near the frontier: how close can they stay to this critical curve? Two different approaches to this question have been developed. We improve upon the best-known bounds in each case, revealing new qualitative features including marked differences between the two approaches.}

\section{Introduction}

A standard branching Brownian motion (BBM) begins with one particle at the origin. This particle moves as a Brownian motion, until an independent exponentially distributed time of parameter 1, at which point it is instantaneously replaced by two new particles. These particles independently repeat the stochastic behaviour of their parent relative to their start position, each moving like a Brownian motion and splitting into two at an independent exponentially distributed time of parameter 1.

Let $N(t)$ be the set of all particles alive at time $t$, and for a particle $v\in N(t)$ let $X_v(s)$ represent its position at time $s\leq t$ (or if $v$ was not yet alive at time $s$, then the position of the unique ancestor of $v$ that was alive at time $s$). If we define
\[M(t) = \max_{v\in N(t)} X_v(t)\]
then it is well known \cite{kolmogorov_et_al:etude_kpp} that 
\[\frac{M(t)}{t}\to \sqrt2 \hs\text{ as } t\to\infty.\]

One of the most striking results on BBM was given by Bramson \cite{bramson:convergence_Kol_eqn_trav_waves}, who calculated fine asymptotics for the distribution of $M(t)$, providing new results on travelling wave solutions to the FKPP equation. Hu and Shi \cite{hu_shi:minimal_pos_crit_mg_conv_BRW} more recently (and for branching random walks rather than BBM) showed fluctuations in the almost sure behaviour of $M(t)$ on the log scale.

The behaviour of the frontier of the system --- loosely speaking, the collection of particles near the maximum $M(t)$ at time $t$ --- is of interest as a tractable model that is conjectured, or in some cases proved, to belong to the same universality class as constructions arising in biology \cite{derrida_simon:survival_prob_brw_absorbing_wall, kolmogorov_et_al:etude_kpp} and statistical physics \cite{bolthausen_et_al:entropic_repulsion, bramson_zeitouni:tightness_max_GFF}. It is natural, then, to ask what the paths of particles near $M(t)$ look like.

The problem of interest in this article is that of consistent maximal displacements: how close can particles stay to the critical line $\sqrt2 u, \hsl u\geq0$? There are at least two ways of making this question precise, each of which has been considered before for the related model of branching random walks. The first is to ask for which curves $f:[0,\infty)\to\R$ it is possible for particles to stay above $f(t)$ for all times $t\geq0$. That is, when is
\[\nu(f):=\P(\forall t\geq0, \exists v\in N(t) : X_v(u)>f(u) \hs\forall u\leq t)\]
non-zero? This was first considered by Jaffuel \cite{jaffuel:crit_barrier_brw_absorption} (for branching random walks), who proved that there is a critical value $A_c = 3^{4/3}\pi^{2/3}2^{-7/6}$ such that if we set $f_a(t) = \sqrt2 t - at^{1/3} - 1$ then $\nu(f_a)>0$ if $a>A_c$, and $\nu(f_a)=0$ if $a<A_c$.

The second approach is to look at recentered paths, specifically the behaviour of
\[\Lambda(t) = \min_{v\in N(t)} \sup_{s\in[0,t]}\{\sqrt2 s-X_v(s)\},\]
as $t\to\infty$. This quantity (or rather, again, its analogue for branching random walks) was studied by Fang and Zeitouni \cite{fang_zeitouni:consistent_maximal_displacement} and by Faraud, Hu and Shi \cite{faraud_hu_shi:conv_biased_rws_trees}, who showed that there is a critical value $a_c = 3^{1/3}\pi^{2/3}2^{-1/2}$ such that almost surely
\[\lim_{t\to\infty}\frac{\Lambda(t)}{t^{1/3}} = a_c.\]

To summarise, the two approaches to the question give similar results: in each case there appears to be a critical line on the $t^{1/3}$ scale above which particles cannot remain. We shall see, however, that if one peers more closely then the two situations are really quite different. Our first result is that not only is $\nu(f_{A_c}) > 0$ (which was previously unknown), but in fact particles may stay far above the curve $f_{A_c}$. Secondly, we are able to give much finer asymptotics for $\Lambda(t)$, both in distribution and almost surely. These developments are redolent of the results of Bramson \cite{bramson:convergence_Kol_eqn_trav_waves} and Hu and Shi \cite{hu_shi:minimal_pos_crit_mg_conv_BRW}. The proofs have certain elements in common with those found in \cite{roberts:simple_path_BBM}, but are decidedly more involved, and we must develop several new techniques along the way.

We now state our three main theorems, which make precise the discussion above.

\begin{thm}\label{jaff}
Let $A_c = 3^{4/3}\pi^{2/3}2^{-7/6}$. Define $g:[0,\infty)\to\R$ by setting
\[g(t) = \sqrt2 t - A_c t^{1/3} + \frac{A_c t^{1/3}}{\log^2(t+e)} - 1.\]
Then
\[\P(\forall t\geq0, \exists v\in N(t) : X_v(u)>g(u) \hs \forall u\leq t) > 0.\]
\end{thm}

\begin{thm}\label{bram}
Let $a_c = 3^{1/3}\pi^{2/3}2^{-1/2}$. Then there exist constants $c_1,c_2>0$ such that for $z \in[1,a_c t^{1/3}/2]$,
\[c_1 ze^{-\sqrt2 z} \leq P(\Lambda(t) \leq a_c t^{1/3} - z) \leq c_2 ze^{-\sqrt2 z}.\]
\end{thm}

Given this result, it is an easy exercise to prove tightness of $\Lambda(t)-a_c t^{1/3}$, and one might suspect that $\Lambda(t) - a_c t^{1/3}$ converges in distribution as $t\to\infty$. This question remains open. However, we are able to show that there are rare times when $\Lambda(t)$ is a long way from $a_c t^{1/3}$.

\begin{thm}\label{hush}
$\Lambda(t)-a_c t^{1/3}$ fluctuates on the logarithmic scale: almost surely
\[\limsup_{t\to\infty} \frac{\Lambda(t)-a_c t^{1/3}}{\log t} =0\]
but
\[\liminf_{t\to\infty} \frac{\Lambda(t)-a_c t^{1/3}}{\log t} =-1/3\sqrt2.\]
\end{thm}

\begin{rmk}
Initially we proved a weaker form of Theorem \ref{bram}, and as a result gained a weaker statement in place of Theorem \ref{hush}. A few months after this article first appeared online, Berestycki, Berestycki and Schweinsberg published an article \cite{berestycki_et_al:critical_bbm_absorp_sp} which revealed the correct asymptotics seen in our new Theorem \ref{bram}. The key to finding the correct answer lies in Proposition 20 of \cite{berestycki_et_al:critical_bbm_absorp_sp}. Other than that, our methods were developed independently. Even given Theorem \ref{bram}, new ideas are required to deduce the stated version of Theorem \ref{hush}.

No upper bound is given in Theorem \ref{jaff}, and the only rigorous upper bound we know of is that given by Jaffuel \cite{jaffuel:crit_barrier_brw_absorption}. In fact, as mentioned above, Jaffuel considered branching random walks, but it is not difficult either to adapt his proof, or to apply his result together with standard tightness properties of Brownian motion, to achieve the same upper bound for BBM. A finer upper bound appears to be difficult to prove.

The above results are stated only for standard BBM. There exist techniques for transferring the proofs to other cases, where for example each particle might give birth to a random number of new particles when it splits. In order to apply our methods we must, however, insist that the distribution of this random number has a finite second moment.
\end{rmk}

\subsection{Layout of the article}
Our main tactic will be to develop detailed estimates for a single Brownian motion, and then to apply standard branching tools to deduce results for the branching system. In Section \ref{jaff_est} we develop our main single-particle estimates, on the probability that a Brownian motion stays within a tube about a function $f$. These will then be used to prove both Theorem \ref{jaff} in Section \ref{jaff_proof} and Theorem \ref{bram} in Section \ref{bram_proof}. Finally, in Section \ref{hush_proof} we apply Theorem \ref{bram}, together with estimates developed in previous sections, to prove Theorem \ref{hush}.

\subsection{Notation}
For $\alpha\in(0,\infty)$ we use the notation $g(t) \asymp_{\alpha} h(t)$ to mean that there exist constants $c(\alpha), C(\alpha)\in (0,\infty)$ depending only on $\alpha$ such that $c(\alpha)g(t)\leq h(t)\leq C(\alpha)g(t)$ for all $t$ in the specified range. Similarly, when we write $g(t) \lesssim_{\alpha} h(t)$ we mean that there exists a constant $C(\alpha)$ depending only on $\alpha$ such that $g(t)\leq C(\alpha) h(t)$ for all $t$ in the specified range. Use of this notation without a specified parameter $\alpha$ ({\em i.e.}~$\asymp$ and $\lesssim$ rather than $\asymp_\alpha$ and $\lesssim_\alpha$) means that the constants are absolute.

For $\gamma\in\R$ and $x>0$, we write $\log^\gamma x$ to mean $(\log x)^\gamma$.

Throughout the article we shall have two twice continuously differentiable functions $f:[0,\infty)\to\R$ and $L:[0,\infty)\to(0,\infty)$, such that $f(0) = -x <0$ and $f(0)+L(0)>0$. We suppose that under $\P$, as well as our BBM we have an independent Brownian motion $\xi_t$, $t\geq0$ started from $0$. Write $(\G_t, t\geq0)$ for the natural filtration of $\xi_t$. We define
\[\ec(f,L,t) = \frac12 \int_0^t f'(s)^2 ds + \int_0^t \frac{\pi^2}{2L(s)^2}ds + f'(t)L(t) + f'(0)f(0) + \frac12 \log L(0) - \frac12 \log L(t).\]

We often assume that there exists a constant $\bar L>0$ such that
\[|L'(0)|L(0) + |L'(t)|L(t) + \int_0^t |L''(s)|L(s) ds + \int_0^t |f''(s)|L(s)ds - |L'(0)|f(0) \leq \bar L\]
for all $t$, which we call assumption (A).

\section{Single-particle estimates}\label{jaff_est}
In this section we are interested in estimating the probability that a Brownian motion stays close to a function $f$. For $\eps>0$ define
\[\rho_L(\eps) = \inf\left\{t>0 : \int_0^t \frac{1}{L(s)^2}ds > \eps\right\}.\]
Our main aim for the section is to prove the following result.

\begin{prop}\label{qasymp}
Assume (A). If $f'(t)\geq0$ then for any $\eps>0$, $t\geq \rho_L(\eps)$ and $0\leq p<q\leq1$,
\begin{multline*}
e^{-\ec(f,L,t)+(1 - q)f'(t)L(t)}\sin\left(\frac{\pi x}{L(0)}\right)\int_p^q \sin(\pi \nu) d\nu\\
\lesssim_{\bar L,\eps}\P\left(\xi_s-f(s)\in (0, L(s)) \hs \forall s\leq t, \hs \xi_t - f(t) \in \left(pL(t),qL(t)\right)\right)\\
\lesssim_{\bar L,\eps} e^{-\ec(f,L,t) + (1-p)f'(t)L(t)}\sin\left(\frac{\pi x}{L(0)}\right)\int_p^q \sin(\pi \nu) d\nu.
\end{multline*}
If $f'(t)\leq0$ then the inequalities are reversed.
\end{prop}

\begin{rmk}
Often we will fix a particular $f$ and $L$, and be interested in particles near $f(t)+L(t)$ at a particular time $t$. For example, if we want to be within a constant distance $C$ of $f(t)+L(t)$, and $L(0),L(t)\geq C$, then both the upper and lower bounds above collapse to
\[\P\left(\xi_s-f(s)\in (0, L(s)) \hs \forall s\leq t, \hs \xi_t - f(t) \in \left(L(t)-C,L(t)\right)\right) \asymp_{\bar L,\eps, C} e^{-\ec(f,L,t)} \frac{1}{L(0)}\frac{1}{L(t)^2}.\]
\end{rmk}

In order to prove Proposition \ref{qasymp}, we begin with a well-known estimate for the probability that a Brownian motion stays within a tube of fixed width for a long time.

\begin{lem}\label{fellerlem}
For any $\eps>0$, for all $t\geq \eps$ and $y\in(-1,1)$, $-1\leq p<q\leq1$,
\[\P(|y+\xi_s|< 1 \hs\forall s\leq t, \hs y+\xi_t \in (p,q)) \asymp_\eps e^{-\pi^2 t/8}\cos\left(\frac{\pi y}{2}\right)\int_p^q \cos\left(\frac{\pi\nu}{2}\right)d\nu.\]
\end{lem}

Lemma \ref{fellerlem} may easily be deduced from \cite[page 342]{feller:book}, or \cite[page 30, Problem 1.7.8]{ito_mckean:diffusion_procs_sample_paths}. Now we want to consider tubes whose width is not fixed. We estimate the probability that $y+\xi_t$ stays within $(-L(s)/2, L(s)/2)$ for all times $s\in[0,t]$, adapting slightly an idea of Novikov \cite{novikov:asymptotic_behaviour_nonexit_probs}.

\begin{lem}\label{pasymp}
Assume (A) holds with $f\equiv 0$. Let $K(s)=L(s)/2$ for each $s\geq0$. Then for any $\eps>0$, $t\geq \rho_L(\eps)$, $y\in(-1,1)$ and $0\leq p<q\leq 1$,
\begin{multline*}
\P(y+\xi_s \in (-K(s), K(s))\hs \forall s\leq t, \hs y+\xi_t \in (2p-1)K(t),(2q-1)K(t)))\\
\asymp_{\bar L,\eps} e^{-\int_0^t \frac{\pi^2}{2L(s)^2}ds + \frac{1}{2}\log L(t) - \frac{1}{2}\log L(0)}\cos\left(\frac{\pi y}{L(0)}\right)\int_p^q \sin(\pi\nu)d\nu.
\end{multline*}
\end{lem}

\begin{proof}
For $z\in(-1,1)$ set
\[U_t = K(t)z + K(t)\int_0^t \frac{1}{K(s)}d\xi_s.\]
Then
\begin{equation}\label{du}
dU_t = d\xi_t + U_t \frac{K'(t)}{K(t)}dt,
\end{equation}
so if we define $\Pt$ by
\[\left.\frac{d\Pt}{d\P}\right|_{\G_t} = e^{\int_0^t \xi_s \frac{K'(s)}{K(s)} d\xi_s - \frac{1}{2}\int_0^t \xi_s^2 \frac{K'(s)^2}{K(s)^2} ds}\]
then by Girsanov's theorem, $(\xi_t,t\geq0)$ under $\Pt$ has the same distribution as $(U_t-U_0,t\geq0)$ under $\P$. 
In particular
\begin{multline*}
\P(|U_s|<K(s)\hsl\forall s\leq t, \hsl U_t \in ((2p-1)K(t),(2q-1)K(t)))\\
 = \Pt(|K(0)z+\xi_s|<K(s)\hsl\forall s\leq t, \hsl K(0)z + \xi_t \in ((2p-1)K(t),(2q-1)K(t))).
\end{multline*}
Thus, letting $z=y/K(0)$,
\begin{align*}
&\E\left[e^{\int_0^t \xi_s \frac{K'(s)}{K(s)} d\xi_s - \frac{1}{2}\int_0^t \xi_s^2 \frac{K'(s)^2}{K(s)^2} ds}\ind_{\{|K(0)z+\xi_s|<K(s) \hsl\forall s\leq t, \hsl K(0)z+\xi_t\in((2p-1)K(t),(2q-1)K(t))\}}\right]\\
&= \Pt(|K(0)z+\xi_s|<K(s) \hs \forall s\leq t, \hs K(0)z+\xi_t\in((2p-1)K(t),(2q-1)K(t)))\\
&= \P(|U_s| < K(s) \hs \forall s\leq t, \hs U_t \in ((2p-1)K(t),(2q-1)K(t)))\\
&= \P\left(\left|z+\int_0^s \frac{1}{K(r)}d\xi_r\right| < 1 \hs \forall s\leq t, \hs z + \int_0^t \frac{1}{K(r)}d\xi_r \in \left(2p-1,2q-1\right)\right)\\
&= \P\left(|z+\xi_s| < 1 \hs \forall s\leq \int_0^t \frac{1}{K(r)^2} dr, \hs z+\xi_{\int_0^t 1/K(r)^2 dr} \in (2p-1,2q-1)\right).
\end{align*}
where for the last line we used the Dubins-Schwarz theorem. Under $\P$, by integration by parts,
\[\frac{K'(t)}{2K(t)}\xi_t^2 = \int_0^t \xi_s \frac{K'(s)}{K(s)} d\xi_s + \int_0^t \frac{\xi_s^2 K''(s)}{2K(s)}ds - \int_0^t \frac{\xi_s^2 K'(s)^2}{2K(s)^2} ds + \frac12 \log K(t) - \frac12 \log K(0)\]
so using assumption (A), on the event $\{|K(0)z+\xi_s|<K(s) \hsl \forall s\leq t\}$,
\[e^{\int_0^t \xi_s \frac{K'(s)}{K(s)} d\xi_s - \frac{1}{2}\int_0^t \xi_s^2 \frac{K'(s)^2}{K(s)^2} ds} \asymp_{\bar L} e^{\frac{1}{2}\log K(0) - \frac{1}{2}\log K(t)}\]
almost surely. Thus
\begin{multline*}
\P(|K(0)z+\xi_s|<K(s) \hsl\forall s\leq t, \hsl K(0)z+\xi_t\in((2p-1)K(t),(2q-1)K(t)))\\
\asymp_{\bar L} e^{\frac{1}{2}\log \frac{L(t)}{L(0)}}\P\left(|z+\xi_s| < 1 \hsl \forall s\leq \int_0^t \frac{1}{K(r)^2} dr, \hsl z+\xi_{\int_0^t 1/K(r)^2 dr} \hspace{-1mm}\in (2p-1,2q-1)\right).
\end{multline*}
The result now follows from Lemma \ref{fellerlem}.
\end{proof}

Finally we apply Girsanov's theorem, together with standard estimates, to consider tubes about the function $f$ (or, to be precise, $f+K$) rather than about $0$.

\begin{proof}[Proof of Proposition \ref{qasymp}]
Define $\hat\P$ by setting
\[\left.\frac{d\hat\P}{d\P}\right|_{\G_t} = e^{\int_0^t (f'(s)+K'(s))d\xi_s - \frac12\int_0^t (f'(s)+K'(s))^2 ds}.\]
For each $t\geq0$ let $\tilde\xi_t = \xi_t - f(t)- K(t)$, and define
\[A_t = \{\tilde\xi_s \in (-K(s),K(s))\hs\forall s\leq t, \hs \tilde\xi_t \in ((2p-1)K(t),(2q-1)K(t))\}.\]
By integration by parts,
\[(f'(t)+K'(t))\tilde \xi_t = (f'(0)+K'(0))\tilde \xi_0 + \int_0^t (f'(s)+K'(s))d\tilde\xi_s + \int_0^t (f''(s)+K''(s))\tilde\xi_s ds\]
so
\begin{align*}
\hat \P(A_t) &= \E\left[e^{\int_0^t (f'(s)+K'(s))d\xi_s - \frac12\int_0^t (f'(s)+K'(s))^2 ds} \ind_{A_t}\right]\\
&= \E\left[e^{\frac12 \int_0^t (f'(s)+K'(s))^2 ds + (f'(0)+K'(0))(f(0)+K(0)) + (f'(t)+K'(t))\tilde \xi_t - \int_0^t (f''(s)+K''(s))\tilde\xi_s ds}\ind_{A_t}\right].
\end{align*}
But on $A_t$, assuming (A), the exponent above is
\[\frac12\int_0^t f'(s)^2 ds + f'(t)K(t) - f'(0)x + f'(t)\tilde\xi_t + O(\bar L)\]
almost surely, so
\begin{equation}\label{hatasymp}
\hat\P(A_t) \asymp_{\bar L} \E\left[e^{f'(t)K(t) + \frac12\int_0^t f'(s)^2 ds - f'(0)x + f'(t)\tilde\xi_t}\ind_{A_t}\right].
\end{equation}
However, by Girsanov's theorem, under $\hat\P$, $(\tilde\xi_t, t\geq0)$ is a Brownian motion started from $x-K(0)$, so by Lemma \ref{pasymp}
\[\hat\P(A_t) \asymp_{\bar L} e^{-\int_0^t \frac{\pi^2}{2L(s)^2}ds + \frac{1}{2}\log L(t) - \frac{1}{2}\log L(0)}\cos\left(\frac{\pi (x-K(0))}{L(0)}\right)\int_p^q \sin(\pi\nu)d\nu.\]
The result now follows from rearranging (\ref{hatasymp}).
\end{proof}

For times $t<\rho_L(\eps)$, the following easy estimate will be sufficient for our needs.

\begin{lem}\label{altqasymp}
Assume (A). For any $t>0$, if $f'(t)\geq 0$, for any $0\leq p\leq q\leq 1$,
\begin{align*}
&\P\left(\xi_s-f(s)\in (0, L(s)) \hs \forall s\leq t, \hs \xi_t - f(t) \in \left(pL(t),qL(t)\right)\right)\\
&\lesssim_{\bar L} e^{-\frac{1}{2}\int_0^t f'(s)^2 ds - f'(0)f(0) - pf'(t)L(t)}\\
&\hspace{25mm}\cdot\left((-f(0))(q^2-p^2)\frac{L(t)^2}{t^{3/2}}\wedge(L(0)+f(0))\big((1-p)^2 - (1-q)^2\big)\frac{L(t)^2}{t^{3/2}}\wedge 1\right).
\end{align*}
\end{lem}

\begin{proof}
Let $\tilde \xi_t = \xi_t - f(t)$. Define
\[B_t = \{\tilde\xi_s > 0 \hs \forall s\leq t, \hs \tilde\xi_t \in (pL(t),qL(t))\}\]
Define $\tilde\P$ by setting
\[\left.\frac{d\tilde\P}{d\P}\right|_{\G_t} = e^{\int_0^t f'(s)d\xi_s - \frac12\int_0^t f'(s)^2 ds}.\]
Note that, using similar integration by parts calculations as in the proof of Proposition \ref{qasymp},
\begin{align*}
\tilde\P(B_t) &= \E\left[e^{\int_0^t f'(s)d\xi_s - \frac12\int_0^t f'(s)^2 ds}\ind_{B_t}\right]\\
&= \E\left[e^{f'(t)\tilde\xi_t - \int_0^t f''(s)\tilde\xi_s ds + f'(0)f(0) + \frac12\int_0^t f'(s)^2 ds}\ind_{B_t}\right]
\end{align*}
so using assumption (A),
\[\P(B_t) \lesssim_{\bar L} e^{-pf'(t)L(t) - f'(0)f(0) - \frac12 \int_0^t f'(s)^2 ds}\tilde\P(B_t).\]
But under $\tilde\P$, by Girsanov's theorem, $(\tilde\xi_t,t\geq0)$ is Brownian motion started from $x$. By the reflection principle,
\begin{align*}
\tilde\P(B_t) &= \frac{1}{\sqrt{2\pi t}}\int_{pL(t)}^{qL(t)} \left(e^{-(x-z)^2/2t} - e^{-(x+z)^2/2t}\right)dz\\
&\lesssim \int_{pL(t)}^{qL(t)} \frac{z}{t^{3/2}} dz \wedge 1 \hs \lesssim \hs x(q^2-p^2)\frac{L(t)^2}{t^{3/2}} \wedge 1.
\end{align*}
This gives us the bound
\[\P(B_t) \lesssim_{\bar L} e^{-\frac{1}{2}\int_0^t f'(s)^2 ds - f'(0)f(0) - pf'(t)L(t)}\left( x(q^2-p^2)\frac{L(t)^2}{t^{3/2}}\wedge 1\right);\]
applying the same calculations, plus assumption (A), to $\hat\xi = f(t)+L(t)-\xi_t$ gives the remaining bound.
\end{proof}

\section{Proof of Theorem \ref{jaff}}\label{jaff_proof}

We now apply Proposition \ref{qasymp} and Lemma \ref{altqasymp} to prove Theorem \ref{jaff}, which said that if we let $A_c = 3^{4/3}\pi^{2/3}2^{-7/6}$, then with positive probability there is always a particle above $\sqrt2 t - A_c t^{1/3} + A_c t^{1/3}/\log^2(t+e) - 1$. Let
\[f(t) = \sqrt2 (t+e) - A_c(t+e)^{1/3} + A_c \frac{(t+e)^{1/3}}{\log^2(t+e)} - C\]
and
\[L(t) = \alpha (t+e)^{1/3} + \beta\frac{(t+e)^{1/3}}{\log(t+e)}\]
where $C$ is any constant such that $f(0)<0$ and $f(0)+L(0)>0$. We will choose $\alpha$ and $\beta$ later. We want to count the number of particles staying above $f(u)$ up to time $t$, but a direct attack via moment estimates will not work because the second moment overcounts the number of particles that go unrealistically high early on. We therefore introduce the upper boundary $f(u)+L(u)$, and count the number of particles that stay in the tube between $f(u)$ and $f(u)+L(u)$ up to time $t$. Even this, though, does not work on its own. Our calculation ends up being dominated by particles which have been to the top of the tube early on and then fallen to the bottom by time $t$, which is not a ``sustainable'' tactic for survival: particles near the bottom of the tube at time $t$ are very likely to fall below $f(u)$ at some time $u$ slightly after $t$. We therefore count only particles that stay within the tube up to time $t$ \emph{and} sit near the top of the tube at time $t$.

To this end, define
\[\tilde N(t) = \#\{v\in N(t) : X_v(s)-f(s) \in(0,L(s)) \hs \forall s\leq t, \hs X_v(t) - f(t) \in (L(t)-2,L(t)-1)\}.\]
We will estimate the first and second moments of $\tilde N(t)$. One may easily check that for our choice of $L$, assumption (A) holds for some constant $\bar L$. Further,
\[\frac12\int_0^t f'(s)^2 ds = t - \sqrt2 A_c t^{1/3} + \sqrt2 A_c \frac{t^{1/3}}{\log^2 t} + O(1)\]
and
\[f'(t)L(t) = \sqrt2\alpha t^{1/3} + \sqrt2\beta \frac{t^{1/3}}{\log t} + O(1).\]
A longer, but still elementary, calculation reveals that
\[\int_0^t \frac{1}{L(s)^2} ds = \frac{3}{\alpha^2}t^{1/3} - \frac{6\beta}{\alpha^3} \frac{t^{1/3}}{\log t} + \frac{9\beta}{\alpha^4}(\beta-2\alpha)\frac{t^{1/3}}{\log^2 t} - \frac{6\beta}{\alpha^5}(2\beta^2 - 9\alpha\beta + 18\alpha^2)\frac{t^{1/3}}{\log^3 t} - E(t)\]
where $E(t) = O(t^{1/3}/\log^4 t)$.

\begin{lem}\label{jaff1mom}
If $\alpha = \beta = 3^{1/3}\pi^{2/3}2^{-1/6}$, then for all $t\geq2$,
\[\E[\tilde N(t)] \asymp \exp\left(\frac{\gamma t^{1/3}}{\log^3 t} + E(t) - \frac12\log t\right)\]
where $\gamma = 2^{1/3}3^{1/3}11\pi^{2/3}$.
\end{lem}

\begin{proof}
From the calculations above, we have
\begin{multline*}
\ec(f,L,t) = t + \left(-\sqrt2 A_c + \frac{3\pi^2}{2\alpha^2}+\sqrt2\alpha\right) t^{1/3}\\
+ \left(\sqrt2 - \frac{3\pi^2}{\alpha^3}\right)\beta\frac{t^{1/3}}{\log t} + \left(\frac{9\pi^2\beta}{2\alpha^4}(\beta-2\alpha) + \sqrt2 A_c\right)\frac{t^{1/3}}{\log^2 t}\\
- \frac{3\pi^2\beta}{\alpha^5}(2\beta^2-9\alpha\beta+18\alpha^2)\frac{t^{1/3}}{\log^3 t} - E(t) - \frac{1}{6}\log t + O(1).
\end{multline*}
Choosing $\alpha = \beta = 3^{1/3}\pi^{2/3}2^{-1/6}$, we see that the coefficients of $t^{1/3}$, $t^{1/3}/\log t$, and $t^{1/3}/\log^2 t$ all disappear, leaving us with
\begin{equation}\label{jaffe}
\ec(f,L,t) = t - \gamma\frac{t^{1/3}}{\log^3 t} - E(t) - \frac{1}{6}\log t + O(1).
\end{equation}
Now applying the many-to-one lemma and Proposition \ref{qasymp}, noting that $\int_{1-2/L(t)}^{1-1/L(t)} \sin(\pi\nu)d\nu \asymp t^{-2/3}$, we see that
\begin{align*}
\E[\tilde N(t)] &= e^t\P(\xi_s-f(s) \in(0,L(s)) \hs \forall s\leq t, \hs \xi_t - f(t) \in (L(t)-2,L(t)-1))\\
&\asymp \exp\left(\frac{\gamma t^{1/3}}{\log^3 t} + E(t) - \frac12 \log t \right).\qedhere
\end{align*}
\end{proof}

\begin{lem}\label{jaff2mom}
If $\alpha = \beta = 3^{1/3}\pi^{2/3}2^{-1/6}$, then for all $t\geq2$,
\[\E[\tilde N(t)^2] \lesssim \E[\tilde N(t)] + \exp\left(\frac{2\gamma t^{1/3}}{\log^3 t} + 2 E(t) - \log t\right)\]
where, again, $\gamma = 2^{1/3}3^{1/3}11\pi^{2/3}$.
\end{lem}

\begin{proof}
We apply the many-to-two lemma, which is a simple tool for calculating second moments (see for example \cite{harris_roberts:many_to_few} for more details). Suppose that $T$ is an independent exponentially distributed random variable of parameter 2 and, given $T$, $(\xi^{(1)}_s,s\geq0)$ and $(\xi^{(2)}_s,s\geq0)$ are standard Brownian motions such that
\begin{itemize}
\item $\xi^{(1)}_s = \xi^{(2)}_s$ for all $s\in[0,T]$;
\item $(\xi^{(1)}_{T+s} - \xi^{(1)}_T, s\geq0)$ and $(\xi^{(2)}_{T+s} - \xi^{(2)}_T, s\geq0)$ are independent given $\G_T$.
\end{itemize}
For $i=1,2$, let
\[A^{(i)}_t = \{\xi^{(i)}_s - f(s) \in (0,L(s))\hs\forall s\leq t\}\]
and
\[C^{(i)}_t = \{\xi^{(i)}_t - f(t) \in (L(t)-2,L(t)-1)\}\]
and define
\[\Theta_t = A_t^{(1)}\cap A_t^{(2)}\cap C_t^{(1)}\cap C_t^{(2)}.\]
Then the many-to-two lemma tells us that
\begin{equation}\label{jaffm22}
\E[\tilde N(t)^2] = \E[\tilde N(t)] + 2\int_0^t e^{2t-s}\P\left(\left.\Theta_t\right|T=s\right)ds.
\end{equation}

We separate $\Theta_t$ into three sections: what happens before time $T$ (when $\xi^{(1)} = \xi^{(2)}$); what happens to $\xi^{(1)}$ after time $T$; and what happens to $\xi^{(2)}$ after time $T$. The latter two are independent given $\xi^{(1)}_T$. The plan is to apply Proposition \ref{qasymp} to each section, although we have to worry about $s$ being too close to $0$ or $t$. To this end define
\[\bar\rho_L(t) = \sup\left\{s<t : \int_s^t \frac{1}{L(u)^2}du > 1\right\}.\]
Note that $\bar\rho_L(t) = t-O(t^{2/3})$ and $\rho_L(1)=O(1)$.

For $s > \rho_L(1)$ and $j\in[0,L(s)-1]$, by Proposition \ref{qasymp} and (\ref{jaffe}), together with the fact that $\int_{1-(j+1)/L(s)}^{1-j/L(s)} \sin(\pi\nu)d\nu \lesssim (j+1)s^{-2/3}$,
\[\P(A^{(1)}_s \cap\{\xi^{(1)}_s - f(s) \in (L(s)-j-1,L(s)-j]\}) \lesssim (j+1)\exp\left( -s + \frac{\gamma s^{1/3}}{\log^3 s} + \sqrt2 j + E(s) - \frac12\log s\right).\]
For $s< \bar\rho_L(t)$, by independence, Proposition \ref{qasymp}, and (\ref{jaffe}),
\begin{align*}
&\P\left( \Theta_t \left| A^{(1)}_s \cap \{\xi^{(1)}_s-f(s) = L(s)-z \}\cap\{T=s\}\right.\right)\\
&= \P\left( A^{(1)}_t \cap C^{(1)}_t \left| A^{(1)}_s \cap \{\xi^{(1)}_s-f(s) = L(s)-z \}\cap\{T=s\}\right.\right)^2\\
&\lesssim (z+1)^2\exp\left( -2(t-s) + 2\gamma\left(\frac{\gamma t^{1/3}}{\log^3 t}- \frac{\gamma s^{1/3}}{\log^3 s}\right) - 2\sqrt2 z + 2E(t) - 2E(s) - \log t -\log s \right).
\end{align*}
Putting these two estimates together, we obtain that for $s\in(\rho_L(1),\bar\rho_L(t))$,
\begin{align*}
\P(\Theta_t | T=s) &\lesssim \sum_j (j+1)^3 s^{-3/2} t^{-1}\exp\left(-2t + s + \frac{2\gamma t^{1/3}}{\log^3 t} - \frac{\gamma s^{1/3}}{\log^3 s} - \sqrt2 j + 2E(t) - E(s)\right)\\
&\asymp s^{-3/2} t^{-1} \exp\left(-2t + s + \frac{2\gamma t^{1/3}}{\log^3 t} - \frac{\gamma s^{1/3}}{\log^3 s} + 2E(t) - E(s)\right).
\end{align*}
Thus
\[\int_{\rho_L(1)}^{\bar\rho_L(t)} e^{2t-s}\P\left(\left.\Theta_t\right|T=s\right)ds \lesssim \exp\left(\frac{2\gamma t^{1/3}}{\log^3 t} + 2E(t) - \log t\right).\]
For $s\leq\rho_L(1)$, we can use the trivial bound $\P(A^{(1)}_s)\leq 1$ together with the above estimate on $\P(\Theta_t |A^{(1)}_s\cap\cdots)$ to get the same bound, so in fact
\begin{equation}\label{lowerint}
\int_{0}^{\bar\rho_L(t)} e^{2t-s}\P\left(\left.\Theta_t\right|T=s\right)ds \lesssim \exp\left(\frac{2\gamma t^{1/3}}{\log^3 t} + 2E(t) - \log t\right).
\end{equation}
When $s\geq\bar\rho_L(t)$, applying Lemma \ref{altqasymp} together with the fact that when $s\geq\bar\rho_L(t)$ we have $t^{1/3}= s^{1/3} + O(1)$, we get that for any $z\in[j,j+1]$,
\[\P(\Theta_t |A^{(1)}_s\cap \{\xi^{(1)}_s-f(s) = L(s)-z \}\cap\{T=s\}) \lesssim \exp\left(-(t-s)-\sqrt2 j\right)\left(\frac{j+1}{(t-s)^{3/2}}\wedge 1\right).\]
Thus in this case, using our previous estimate on $\P(A^{(1)}_s \cap\{\xi^{(1)}_s - f(s) \in (L(s)-j-1,L(s)-j]\})$,
\[\P(\Theta_t|T=s) \lesssim \exp\left(-2t+s + \frac{\gamma t^{1/3}}{\log^3 t} + E(t)\right)\frac{1}{s^{1/2}(t-s+1)^{3/2}}\]
and so (since for $s\geq \bar\rho_L(t)$ we have $s^{-1/2}\asymp t^{-1/2}$) we get
\[\int_{\bar\rho_L(t)}^t e^{2t-s}\P\left(\left.\Theta_t\right|T=s\right)ds \lesssim \exp\left(\frac{\gamma t^{1/3}}{\log^3 t} + E(t) - \frac12\log t\right)\lesssim \E[\tilde N(t)].\]
Putting this together with (\ref{lowerint}) and plugging back into (\ref{jaffm22}), we get the desired bound.
\end{proof}

\begin{proof}[Proof of Theorem \ref{jaff}]
It is clearly enough to prove that with probability bounded away from zero there is always a particle above $f(t)$.
Note that, by Cauchy-Schwarz and Lemmas \ref{jaff1mom} and \ref{jaff2mom},
\[\P(\exists v\in N(t) : X_v(s) > f(s) \hs \forall s\leq t) \geq \P(\tilde N(t)\geq1) \geq \frac{\E[\tilde N(t)]^2}{\E[\tilde N(t)^2]} \geq c\]
for some constant $c$, proving the theorem.
\end{proof}

\section{Proof of Theorem \ref{bram}}\label{bram_proof}
For $x=x_t>0$, we now want to estimate
\[\P(\exists v\in N(t) : X_v(u) > \sqrt2 u -x \hs \forall u\leq t).\]
The general tactic is to consider instead the set of particles that remain between $\sqrt2 u-x$ and $\sqrt2 u-x + L_t(u)$ for $u\in[0,t]$ for some function $L_t$. As in Section \ref{jaff_proof}, the reason for this restriction is that with no upper boundary imposed, events of vanishing probability have a distorting effect on the moments of the number of particles with large position. However, we have not succeeded in finding an appropriate $L_t$; it turns out that for most seemingly sensible choices, the probability of staying within the $L_t$-tube is significantly smaller than the probability of staying above $\sqrt2 u-x$ until time $t$. Instead we look at the probability of hitting the top of an appropriate tube, and show that if we hit the top of the tube then we stay above $\sqrt2 u-x$ up to time $t$ with positive probability. This last observation was made by Berestycki, Berestycki and Schweinsberg \cite[Proposition 20]{berestycki_et_al:critical_bbm_absorp_sp}; for completeness, and since their notation is different from ours, we include a proof translated into our own setup. This will be our Proposition \ref{ppp}. All other estimates are independent of those in \cite{berestycki_et_al:critical_bbm_absorp_sp}.

Throughout this section we will use the functions
\[f_{t,z}(u) = \sqrt2 u - a(t+1)^{1/3} + z\]
and
\[L_t(u) = a(t+1-u)^{1/3}\]
where $a=a_c = 3^{1/3}\pi^{2/3}2^{-1/2}$. Sometimes we leave out the subscripts, or write $f_z$ instead of $f_{t,z}$, where the parameters are clear from context.
Let
\[p(t,z;s,y) = \P(\xi_u - f_{t,z}(u)\in(0,L_t(u))\hs\forall u\leq s, \hs \xi_s - f_{t,z}(s) \in [L_t(s)-y-1,L_t(s)-y)).\]
We also define
\begin{equation}\label{qdef}
q(t,z;s,y) = yze^{-s-\sqrt2 z + \sqrt2 y}t^{-1/2}(t+1-s)^{-1/2}
\end{equation}
and
\[\tilde q(z;s,y) = e^{-s-\sqrt2 z + \sqrt2 y}(yzs^{-3/2}\wedge 1).\]
We begin by applying Proposition \ref{qasymp} and Lemma \ref{altqasymp} to our particular choice of $f$ and $L$.

\begin{lem}\label{1pb}
If $s\geq t^{2/3}$, $z\in[1,at^{1/3}/2]$ and $y\in[1,a(t-s)^{1/3}/2]$ then
\[p(t,z;s,y) \asymp q(t,z;s,y).\]
If $s\geq t^{2/3}$, then for any $z$ and $y$,
\[p(t,z;s,y) \lesssim q(t,z;s,y+1).\]
If $s\leq t^{2/3}$, then for any $y,z\geq0$,
\[p(t,z;s,y) \lesssim \tilde q(z;s,y).\]
\end{lem}

\begin{proof}
We note that
\[\int_0^s \frac{\pi^2}{2L(u)^2}du = \sqrt2 L(0) - \sqrt2 L(s),\]
while
\[\frac{1}{2}\int_0^s f'(u)^2 du = s;\]
thus
\[\ec(f,L,s) = s + \sqrt2 z + \frac16 \log (t+1) - \frac16 \log(t+1-s) + O(1)\]
and it is simple to check that assumption (A) holds. Proposition \ref{qasymp} then tells us that for $s\geq t^{2/3}$,
\begin{multline*}
p(t,z;s,y)\\
\asymp e^{-s - \sqrt2 z + \sqrt2 y} (t+1)^{-1/6}(t+1-s)^{1/6}\sin\left(\frac{\pi z}{a(t+1)^{1/3}}\right)\frac{1}{(t+1-s)^{1/3}}\sin\left(\frac{\pi (y+1)}{a(t+1-s)^{1/3}}\right)
\end{multline*}
and the first two claims follow.
For the final claim, we instead apply Lemma \ref{altqasymp}: since for $s\leq t^{2/3}$, we have $L(0) - L(u) = O(1)$, we get
\[p(t,z;s,y) \lesssim e^{-s - \sqrt2 z + \sqrt2 y}(yz s^{-3/2} \wedge 1) = \tilde q(z;s,y).\qedhere\]
\end{proof}

We now work towards counting the number of particles at the top of the tube, again by estimating moments. Let
\begin{multline*}
M_{t,z}(u) = \{v\in N(u) : X_v(r)-f_{t,z}(r) \in (0,L_t(r)) \hsl\forall r\leq u,\\
X_v(u)-f_{t,z}(u) \in [L_t(u)-2,L_t(u)-1)\}.
\end{multline*}
We also abuse notation by writing $\xi\in M_{t,z}(u)$ if $\xi_r-f_{t,z}(r) \in (0,L_t(r))$ $\forall r\leq u$ and $\xi_u - f_{t,z}(u)\in[L_t(u)-2,L_t(u)-1)$.

\begin{lem}\label{1mom}
If $u\geq t^{2/3}$, then
\[\E[\#M_{t,z}(u)] \asymp ze^{-\sqrt2 z}t^{-1/2}(t+1-u)^{-1/2}.\]
\end{lem}

\begin{proof}
By the many-to-one lemma,
\[\E[\#M_{t,z}(u)] = e^u \P(\xi\in M_{t,z}(u)) = e^u p(t,z;u,1).\]
The result now follows from Lemma \ref{1pb}.
\end{proof}

\begin{lem}\label{2mom}
If $m,n\in[t/3,2t/3]$ and $z\geq1$,
\[\E[(\#M_{t,z}(m))(\#M_{t,z}(n))] \leq ze^{-\sqrt2 z}t^{-1}\Big(t^{-1} + t^{-1/2}(1+|n-m|)^{-1/2} + (1+|n-m|)^{-3/2}\Big).\]
(Note that $m$ and $n$ need not necessarily be integers here, although they will be integers when we apply this result.)
\end{lem}

\begin{proof}
First suppose that $m\leq n$. We apply the many-to-two lemma (specifically Example 6 of \cite{harris_roberts:many_to_few}), which tells us that
\begin{multline*}
\E[(\#M_{t,z}(m))(\#M_{t,z}(n))]= e^n \P(\xi \in M_{t,z}(m)\cap M_{t,z}(n))\\
+ 2\int_0^m e^{n+m-r}\P(\xi^{(1)} \in M_{t,z}(m), \hsl \xi^{(2)}\in M_{t,z}(n)|T=r)dr.
\end{multline*}
where $\xi^{(1)}$ and $\xi^{(2)}$ are the dependent Brownian motions defined in the proof of Lemma \ref{jaff2mom}, and $T$ is their split time. Now, integrating out the value of $\xi^{(1)}_r$ and then applying Lemma \ref{1pb}, if $r\in[t^{2/3},m-t^{2/3}]$,
\begin{multline*}
\P(\xi^{(1)} \in M_{t,z}(m), \hsl \xi^{(2)}\in M_{t,z}(n)|T=r)\\
\asymp \int_0^{L_t(r)} p(t,z;r,y)p(t-r,y;m-r,1)p(t-r,y;n-r,1)dy\\
\lesssim \int_0^{L_t(r)} q(t,z;r,y+1)q(t-r,y;m-r,2)q(t-r,y;n-r,2)dy.
\end{multline*}
Recalling the definition (\ref{qdef}) of $q$, an easy calculation reveals that for $r\in[t^{2/3},m-t^{2/3}]$,
\[\P(\xi^{(1)} \in M_{t,z}(m), \hsl \xi^{(2)}\in M_{t,z}(n)|T=r) \lesssim z e^{-m-n+r - \sqrt2 z}t^{-3}.\]
If $r\leq t^{2/3}$, then we replace $q(t,z;r,y)$ with $\tilde q(z;r,y)$ and get
\[\P(\xi^{(1)} \in M_{t,z}(m), \hsl \xi^{(2)}\in M_{t,z}(n)|T=r) \lesssim z e^{-m-n+r - \sqrt2 z}(r+1)^{-3/2}t^{-2}.\]
If $r\in [m-t^{2/3},n-t^{2/3}]$, then we instead replace $q(t-r,y;m-r,1)$ with $\tilde q(y;m-r,1)$ and get
\[\P(\xi^{(1)} \in M_{t,z}(m), \hsl \xi^{(2)}\in M_{t,z}(n)|T=r) \lesssim z e^{-m-n+r - \sqrt2 z}(m+1-r)^{-3/2}t^{-2}.\]
If $r\geq (n-t^{2/3})$, then we replace both the latter $q$ factors with the appropriate $\tilde q$ and get
\[\P(\xi^{(1)} \in M_{t,z}(m), \hsl \xi^{(2)}\in M_{t,z}(n)|T=r) \lesssim z e^{-m-n+r - \sqrt2 z}t^{-1}(m+1-r)^{-3/2}(n+1-r)^{-3/2}.\]
Putting these estimates together, we obtain
\[\int_0^m e^{n+m-r}\P(\xi^{(1)} \in M_{t,z}(m), \hsl \xi^{(2)}\in M_{t,z}(n)|T=r)dr \lesssim ze^{-\sqrt2 z}\big(t^{-2} + t^{-1}(1+n-m)^{-3/2}\big).\]
We also have that when $m\leq n-t^{2/3}$,
\begin{align*}
\P(\xi\in M_{t,z}(m)\cap M_{t,z}(n)) &\asymp q(t,z;m,1)q(t-m,1;n-m,1)\\
&\lesssim ze^{-n-\sqrt2 z}t^{-3/2}(n+1-m)^{-1/2}
\end{align*}
and when  $m\in(n-t^{2/3},n]$,
\begin{align*}
\P(\xi\in M_{t,z}(m)\cap M_{t,z}(n)) &\lesssim q(t,z;m,1)\tilde q(1;n-m,1)\\
&\lesssim ze^{-n-\sqrt2 z}t^{-1}(n+1-m)^{-3/2}.
\end{align*}
By symmetry we get similar results when $m>n$, and obtain the stated result.
\end{proof}

We want to count the number of particles hitting the top of the $L$-tube whose ancestors have never hit either boundary of the tube. In fact, to get a lower bound we will restrict to those hitting the top of the tube between times $t/3$ and $2t/3$. For a particle $v$, let $\sigma_v$ be its birth time and $\tau_v$ its time of death. Let
\begin{multline*}
J_{t,z} = \Big\{ v \in \bigcup_{u\geq0} N(u) : \exists s\in [\sigma_v , \tau_v)\cap [t/3,2t/3) \hbox{ with }  X_v(s) - f_{t,z}(s) = L_t(s),\\
X_v(u) - f_{t,z}(u) \in (0,L_t(u)) \hs \forall u\leq s \Big\}.
\end{multline*}
Otherwise said, if we imagine particles being absorbed if they hit either boundary of the tube, then $J_{t,z}$ is the number of particles absorbed at the top of the tube between times $t/3$ and $2t/3$.

\begin{lem}\label{Mlem}
For $z\in[0, a(t+1)^{1/3}/2)$,
\[\P(J_{t,z}\neq\emptyset) \gtrsim z e^{-\sqrt2 z}.\]
\end{lem}

\begin{proof}
Define
\[\mathcal N_t = \mathbb N \cap [t/3,2t/3)\]
and
\[\tilde J_{t,z} = \sum_{j\in \mathcal N_t}\# M_{t,z}(j).\]
It is easy to see that
\begin{equation}\label{Ms}
\P(J_{t,z}\neq\emptyset) \gtrsim \P(\tilde J_{t,z}\geq 1 ).
\end{equation}
By Cauchy-Schwarz,
\[\P(\tilde J_{t,z}\geq 1 ) \geq \frac{\E[\tilde J_{t,z}]^2}{\E[\tilde J_{t,z}^2]}.\]
Applying Lemma \ref{2mom},
\begin{align*}
\E[\tilde J_{t,z}^2] & = \sum_{m,n\in\mathcal N_t} \E\left[(\#M_{t,z}(m))(\#M_{t,z}(n))\right]\\
&\lesssim ze^{-\sqrt2 z}t^{-1}\sum_{m,n\in \mathcal N_t} \Big(t^{-1}+t^{-1/2}(1+n-m)^{-1/2}+(1+n-m)^{-3/2}\Big) \hsl \lesssim \hsl ze^{-\sqrt2 z}.
\end{align*}
But by Lemma \ref{1mom},
\[\E[\tilde J_{t,z}] \asymp \sum_{j\in \mathcal N_t} ze^{-\sqrt2 z}t^{-1} \asymp ze^{-\sqrt2 z},\]
so
\[\P(\tilde J_{t,z}\geq 1 ) \gtrsim \frac{(ze^{-\sqrt2 z})^2}{ze^{-\sqrt2 z}} = ze^{-\sqrt2 z}\]
as required.
\end{proof}

We have established that the probability of a particle hitting the top of the $L_t$-tube at some time $s\in[t/3,2t/3]$, and having never gone below $f_{t,z}$, behaves as we would like. We now wish to show that such a particle has a positive probability, independent of $t$ and $s$, of staying above $f_{t,z}$ up to time $t$. For this we follow \cite{berestycki_et_al:critical_bbm_absorp_sp}. The idea is this: suppose that a particle $v$ has hit the top of the tube. Then $v$ will, with reasonable probability, have a large number of descendants just below the top of the tube a short time later. Each of these descendants then has a reasonable probability of hitting the top of the tube again. Each descendant of $v$ that hits the top of the tube again we call a {\em tube child} of $v$. We show that this concept of tube children can be used to build a family tree very much like a Galton-Watson process, and calculate that the associated offspring distribution has mean larger than 1. Such a Galton-Watson process survives with strictly positive probability, and the survival of the family tree entails that some particle stays above $f_{t,z}$ up to time $t$. We now give details of this heuristic.

We use the following result of Neveu \cite{neveu:multiplicative_martingales_spatial_bps} to quantify the statement that a particle which hits the top of the tube will, with reasonable probability, have a large number of descendants just below the top of the tube a short time later. Define
\[K(y) = \#\bigg\{ v \in \bigcup_{t\geq0} N(t) : X_v(u) > \sqrt2 u - y \hs \forall u<\sigma_v, \hs \exists s\in [\sigma_v,\tau_v) \hbox{ with } X_v(s) = \sqrt2 s - y \bigg\}.\]
That is, $K(y)$ counts the number of particles that hit the line $(\sqrt2 s - y, s\geq0)$ whose ancestors have never hit the same line. In other words, if we imagine particles being absorbed when they hit the line, then $K(y)$ counts the total number of particles absorbed. Later it will also be useful to consider
\[K(y,t) = \#\bigg\{ v \in \bigcup_{r\geq0} N(r) : X_v(u) > \sqrt2 u - y \hs \forall u<\sigma_v, \hs \exists s\in [\sigma_v,\tau_v\wedge t) \hbox{ with } X_v(s) = \sqrt2 s - y \bigg\},\]
the number of particles in $K(y)$ absorbed before time $t$.

\begin{lem}[Neveu \cite{neveu:multiplicative_martingales_spatial_bps}]\label{neveu}
There exists a random variable $W$ taking values in $(0,\infty)$ with $\E[W]=\infty$ such that
\[ye^{-\sqrt2 y} K(y) \to W\]
almost surely as $y\to\infty$.
\end{lem}

We now build our family tree and carry out the rest of the above heuristic. Again we stress that the proof of Proposition \ref{ppp} is based on \cite[Proposition 20]{berestycki_et_al:critical_bbm_absorp_sp}.

\begin{prop}\label{ppp}
There exists $\delta>0$ such that for all large $t$,
\[\P(\exists v\in N(t): X_v(s) > \sqrt2 s - at^{1/3} \hs \forall s\leq t) \geq \delta.\]
\end{prop}

\begin{proof}
We construct a sequence of sets $Z_0,Z_1,\ldots$ whose sizes will behave very much like a Galton-Watson process.

By Lemma \ref{Mlem}, we can choose $t_0>0$, $\eps>0$ such that for all $t\geq t_0$,
\[\P(J_{t,z} \neq \emptyset) \geq \eps z e^{-\sqrt2 z} \hs \forall z\in[0,at^{1/3}/2).\]
By Lemma \ref{neveu} and the fact that $K(z,t)$ is increasing in $z$ and $t$, we can choose $z_1>0$, $t_1>0$ such that
\[\E[K(z_1,t_1)] > \frac{e^{\sqrt2 z_1}}{\eps z_1}.\]
Fix $t\geq t_0$. By increasing $t_0$ if necessary, we may also assume that $a(t_0+t_1)^{1/3} - z_1 < at_0^{1/3}$ and $L(t-t_0)>z_1$. We say that a particle \emph{succeeds} if it stays above the line $(\sqrt2 u - at^{1/3}, \hsl u\geq 0)$ up to time $t-t_0-t_1$.

Let $Z_0 = \{(v_0, 0)\}$ where $v_0$ is the initial particle. We now recursively define sets $Y_j$ and $Z_{j+1}$ for each $j\geq0$ (we will give an intuitive description immediately afterwards). Given $(v,s)\in Z_j$, let $Y^v_j$ consist of the set of ordered pairs $(v',s')$ such that $v'$ is a descendant of $v$ that hits the line $(\sqrt2 u - at^{1/3} + a(t-s)^{1/3}-z_1, \hsl u\geq s)$ at some time $u\in [s,s+t_1)$, and $s'$ is the first such time. Let
\[Y_j = \bigcup_{\substack{v: \exists s \text{ with }\\ (v,s)\in Z_j}} Y^v_j.\]
For $(v,s)\in Y_j$, let $Z^v_{j+1}$ consist of the set of ordered pairs $(v',s')$ such that $v'$ is a descendant of $v$ that hits the curve $(\sqrt2 u - at^{1/3} + a(t-u)^{1/3}, u\geq s)$ at some time $u\in[s+(t-s)/3, (s+2(t-s)/3)\wedge(t-t_0-t_1))$, $s'$ is the first such time, and $v'$ stayed above $\sqrt2 u - at^{1/3}$ for all $u\in[s,s')$. Let
\[Z_{j+1} = \bigcup_{\substack{v: \exists s \text{ with }\\ (v,s)\in Y_j}} Z^v_{j+1}.\]

In words, $Z_j$ contains particles that hit the top of the tube at appropriate times, and $Y_j$ contains descendants of those particles that fall distance $z_1$ below the top of the tube shortly afterwards.  Up until a particle succeeds, we see that each particle in $Z_j$ has a number of descendants in $Y_j$ with distribution $K(z_1,t_1)$, and each particle in $Y_j$ has at least one descendant in $Z_{j+1}$ with probability at least $\eps z_1 e^{-\sqrt2 z_1}$.

To complete the proof, on an auxiliary probability space let
\[\eta = \sum_{j=1}^K B_j\]
where $K, B_1, B_2, \ldots$ are independent random variables, $K$ is distributed like $K(z_1,t_1)$, and each $B_j$ is Bernoulli with parameter $\eps z_1 e^{-\sqrt2 z_1}$. Let $Z_0',Z_1',\ldots$ be a Galton-Watson process with offspring distribution $\eta$. Then since
\[\E[\eta] = \E[K]\eps z_1 e^{-\sqrt2 z_1} > 1,\]
the process $Z'$ survives forever with positive probability; that is,
\[\delta' := \lim_n \P(Z_n' > 0) > 0.\]
But as described above, either a particle succeeds or the size of $Z_n$ stochastically dominates $Z_n'$; otherwise said,
\[\P(Z_n = \emptyset) \leq \P(\exists v\in N(t-t_0-t_1) : X_v(u) > \sqrt2 u - at^{1/3} \hs \forall u\leq t- t_0-t_1) + \P(Z_n' = 0).\]
Clearly by construction $Z_n$ is eventually empty, so the left-hand side above converges to 1, but $\P(Z_n'=0)\to 1-\delta'$, so we must have
\[\P(\exists v\in N(t-t_0-t_1) : X_v(u) > \sqrt2 u - at^{1/3} \hs \forall u\leq t- t_0-t_1) \geq \delta'.\]
Now it is a simple task to show that if we stay above $\sqrt2 u - at^{1/3}$ up to $t-t_0-t_1$, we have a strictly positive probability (not depending on $t$) of doing so until time $t$. Indeed, recall that $v_0$ is our initial particle, and note that since $\P(\inf_{s\leq 1} X_{v_0}(s) \geq -1, X_{v_0}(1) \geq \sqrt2) > 0$, there exists $\delta''>0$ such that
\[\P(\exists v\in N(t-t_0-t_1) : X_v(u) > \sqrt2 u - at^{1/3} + 1 \hs \forall u\leq t- t_0-t_1) \geq \delta''.\]
But any particle above $\sqrt2 (t-t_0-t_1) - at^{1/3} + 1$ at time $t-t_0-t_1$ has probability at least
\[\P(\xi_u > \sqrt2 u - 1 \hs\forall u\leq t_0+t_1)\]
of staying above $\sqrt2 u - at^{1/3}$ for all $u\in[t-t_0-t_1,t]$. Since this does not depend on $t$, we get the result.
\end{proof}

The proof of Theorem \ref{bram} is now straightforward.

\begin{proof}[Proof of Theorem \ref{bram}]
Recall that
\[L(s) = L_t(s) = a(t+1-s)^{1/3},\]
\[f(s) = f_{t,z}(s) = -a(t+1)^{1/3} + z + \sqrt2 s,\]
and
\begin{multline*}
J_{t,z} = \Big\{ v \in \bigcup_{u\geq0} N(u) : \exists s\in [\sigma_v \vee t/3, \tau_v \wedge 2t/3) \hbox{ with }  X_v(s) = f_{t,z}(s) + L_t(s),\\
X_v(u) \in (f_{t,z}(u),f_{t,z}(u)+L_t(u)) \hs \forall u\leq s \Big\}.
\end{multline*}

We begin with the lower bound. The probability that a particle stays above $f_{t,z}(u)$ up to time $t$ is at least the probability that a particle $v$ hits $f_{t,z}(s)+L_t(s)$ at some time $s\in[t/3,2t/3)$, having stayed above $f_{t,z}(u)$ for all $u\leq s$, and then a descendant of $v$ stays above $f_{t,z}(u)$ for all times $u\in [s,t]$. Applying the Markov property,
\begin{multline*}
\P(\exists v\in N(t) : X_v(u) > f_{t,z}(u) \hs \forall u\leq t)\\
\geq \P(J_{t,z} \neq \emptyset) \inf_{s\in[t/3,2t/3)}\P(\exists v\in N(t-s) : X_v(u) > \sqrt2 u - a(t-s)^{1/3} \hs \forall u\leq t-s).
\end{multline*}
Lemma \ref{Mlem} tells us that $\P(J_{t,z} \neq \emptyset) \gtrsim ze^{-\sqrt2 z}$, and Proposition \ref{ppp} tells us that the latter probability is bounded below by some fixed $\delta>0$.
We deduce that
\[\P(\exists v\in N(t) : X_v(u) > f_{t,z}(u) \hs \forall u\leq t) \gtrsim ze^{-\sqrt2 z}\]
which is our desired lower bound.

For an upper bound, in order to stay above $f_{t,z}(s)$ for all $s\leq t$, clearly a particle must either hit $f_{t,z}(s)+L_t(s)$ for some $s\leq t$, or stay between $f_{t,z}(s)$ and $f_{t,z}(s)+L_t(s)$ for all $s\leq t$. That is,
\begin{multline*}
\P(\exists v\in N(t) : X_v(s) > f_{t,z}(s) \hsl \forall s\leq t)\\
\leq \sum_{j=0}^{\lfloor t\rfloor} \P\bigg(\exists v\in N(j+1) : X_v(u)-f_{t,z}(u)\in (0,L_t(u))\hsl\forall u\leq j, \hsl \sup_{s\in[j,j+1]} X_v(s)> f_{t,z}(j)+L_t(j)\bigg)\\
+ \P(\exists v\in N(t) : X_v(u) - f_{t,z}(u) \in (0,L_t(u))\hsl \forall u\leq t).
\end{multline*}
By the many-to-one lemma and the Markov property, this is at most
\[\sum_{j=0}^{\lfloor t\rfloor} \int_0^{L_t(j)} e^{j+1} p(t,z;j,y)\P\bigg(\sup_{s\in[0,1]}\xi_s \geq y\bigg)dy + e^t p(t,z;t,0).\]
We now apply Lemma \ref{1pb} to see that
\begin{align*}
&\P(\exists v\in N(t) : X_v(s) > f_{t,z}(s) \hsl \forall s\leq t)\\
&\lesssim \sum_{j=0}^{\lfloor t^{2/3}\rfloor}\int_0^{L_t(j)} e^{j+1} (y+1) ze^{-j-\sqrt2 z + \sqrt2 y} (j+1)^{-3/2} e^{-y^2/2} dy\\
&\hspace{10mm} + \sum_{j=\lceil t^{2/3}\rceil}^{\lfloor t\rfloor}\int_0^{L_t(j)} e^{j+1} (y+1) ze^{-j-\sqrt2 z + \sqrt2 y} t^{-1/2}(t+1-j)^{-1/2} e^{-y^2/2} dy\\
&\hspace{20mm} + e^t ze^{-t-\sqrt2 z}t^{-1/2}\\
&\lesssim ze^{-\sqrt2 z} \sum_{j=0}^{\lfloor t^{2/3}\rfloor} (j+1)^{-3/2} + ze^{-\sqrt2 z}t^{-1/2} \sum_{j=\lceil t^{2/3}\rceil}^{\lfloor t\rfloor} (t+1-j)^{-1/2} + ze^{-\sqrt2 z} t^{-1/2}\\
&\lesssim ze^{-\sqrt2 z}
\end{align*}
which completes the proof.
\end{proof}

\section{Proof of Theorem \ref{hush}}\label{hush_proof}

We recall the setup of Theorem \ref{hush}. For $v\in N(t)$, we let $\lambda(v,t) = \sup_{s\in[0,t]}\{\sqrt2 s - X_v(s)\}$ and define $\Lambda(t) = \min_{v\in N(t)} \lambda(v,t)$. Then we wish to show that almost surely,
\[\limsup_{t\to\infty} \frac{\Lambda(t)-a t^{1/3}}{\log t} =0\]
but
\[\liminf_{t\to\infty} \frac{\Lambda(t)-a t^{1/3}}{\log t} =-1/3\sqrt2,\]
where $a=a_c=3^{1/3}\pi^{2/3}2^{-1/2}$.

Showing that $\limsup(\Lambda(t)-at^{1/3})/\log t = 0$ is not difficult, simply by applying the estimates from Theorem \ref{bram} together with standard branching arguments and Borel-Cantelli exactly as in \cite{roberts:simple_path_BBM}. However the proof that $\liminf(\Lambda(t)-at^{1/3})/\log t = -1/3\sqrt2$ requires a novel approach. We begin with a heuristic explanation before starting on the details.

Since the probability of staying above $f_t(u):=\sqrt2 u - at^{1/3} + \frac{1}{\sqrt2}\log t$ up to time $t$ is approximately $1/t$, we may naively hope that the probability that a particle stays above $f_t(u)$ for all $u\leq t$ for some $t\in[n,2n)$ is of constant order (does not decay in $n$). A geometric trials argument would then suggest that the event occurs infinitely often. This is exactly the approach that works when looking at the position of the maximal particle: see \cite{hu_shi:minimal_pos_crit_mg_conv_BRW} or \cite{roberts:simple_path_BBM}. We might therefore begin by estimating moments of
\begin{equation}\label{Hdef}
\int_n^{2n} \# H_t dt,
\end{equation}
where $H_t$ is something like ``the set of particles that stay above $f_t(u)$ for all $u\leq t$''. However, we know as in previous sections that the second moments obtained in such a calculation will be too large, and so, following our strategy from Section \ref{bram_proof}, we might replace $H_t$ with ``the set of particles that hit the top of the $L_t$-tube at some time $u\in[t/3,2t/3]$'' for $L_t(u) = a(t-u)^{1/3}$. But this approach still yields second moments that are too large. This is a clue that something different is happening, and we will need an alternative strategy.

The key is to realise that if a particle manages to stay above $f_t(u)$ up to time $t$, it will have stayed above $f_t(u) + \delta$ for some $\delta>0$. Thus if $s\leq t$ and $f_t(0)-f_s(0)<\delta$, the same particle will have stayed above $f_s(u)$ up to time $s$. But there exists $\eta>0$ such that $f_t(0)-f_s(0)<\delta$ for all $s\in [t-\eta t^{2/3},t]$. We see, therefore, that the value of $\int_n^{2n} \# H_t dt$ is entirely misleading. We should instead work with something like
\begin{equation}\label{H2def}
\sum_{k=1}^{n^{1/3}} \# H_{n+kn^{2/3}}.
\end{equation}
But with the choice of $f_t$ above, this quantity will decay like $n^{-2/3}$, and we realise that we have been working with the wrong function $f_t$ all along. Instead we should choose
\[f_t(u) = \sqrt2 u - at^{1/3} + \frac{1}{3\sqrt2}\log t,\]
in which case (\ref{H2def}) will be of constant order. We show that its first two moments behave well and deduce that, infinitely often, there are particles that stay above the line $f_t(u)$, $u\in[0,t]$.

To see that there are no particles above $g_t(u) := \sqrt2 u - at^{1/3} + \frac{(1+\eps)}{3\sqrt2}\log t$ for $\eps>0$, we return to a quantity more like (\ref{Hdef}). Its expected value is large, but given that a particle stays above $g_t(u)$ for all $u\leq t$, it is even larger. We use this observation to complete the proof. (There are other possible approaches. For example, we could work with something like (\ref{H2def}) again, but then we would have to worry about times in $(n+kn^{2/3},n+(k+1)n^{2/3})$ for each $k$. This could be done fairly easily but would rely on some technical estimates. Using (\ref{Hdef}) will be slicker.)

We split the proof of Theorem \ref{hush} into four lemmas, each of which represents an upper or a lower bound for a $\limsup$ or a $\liminf$. We begin with the easier two.

\subsection{Bounds on the $\limsup$}

\begin{lem}\label{noone_io}
\[\limsup_{t\to\infty} \frac{\Lambda(t) - a t^{1/3}}{\log t} \geq 0 \hs \hbox{ almost surely}.\]
\end{lem}

\begin{proof}
To rephrase the statement of the lemma, we show that for any $\varepsilon\in(0,1)$ there are arbitrarily large times such that no particles have stayed above $\sqrt{2}u - at^{1/3} + \eps\log t$ for all $s\leq t$. Choose $\delta < \varepsilon/2$, let $t_1 = 1$ and for $n>1$ let $t_n = \exp\left(\frac{1}{\delta}\exp(2t_{n-1})\right)$. Define
\[E_n = \{\exists v\in N(t_n) : X_v(u) > \sqrt{2}u - at_n^{1/3} + \eps\log t_n \hs \forall u\leq t_n\}\]
and
\[F_n = \{|N(t_n)|\leq e^{2t_n}, \hs |X_v(t_n)| \leq \sqrt{2}t_n \hs \forall v\in N(t_n)\}.\]
We know that $F_n$ occurs for all large $n$, so it suffices to show that
\[\P\bigg(\bigcap_{k\geq n} (E_k\cap F_k)\bigg) = \lim_{N\to\infty} \prod_{k=n}^N \P\bigg(E_k\cap F_k \bigg| \bigcap_{j=n}^{k-1}(E_j\cap F_j)\bigg) = 0 \hbox{ forall } n\geq0.\]
For a particle $v$, let $N^v(t)$ be the set of descendants of $v$ at time $t$, and let $E^v_n$ be the event that some descendant of $v$ at time $t_n$ has stayed above $\sqrt{2}u - a t_n^{1/3} - \eps\log t_n$ for all times $u\leq t_n$. Also let $s_n = t_n-t_{n-1}$. Then if $v\in N(t_{n-1})$ and $X_v(t_{n-1})\leq \sqrt2 t_{n-1}$,
\begin{align*}
\P(E^v_n | \F_{t_{n-1}})&= \P(\exists w\in N^v(t_n) : X_w(u)>\sqrt2 u - at_n^{1/3} + \eps\log t_n \hs\forall u\leq t_n | \F_{t_{n-1}})\\
&= \P(\exists w\in N(s_n) : X_w(u)>\sqrt2 u - a t_n^{1/3} + \eps\log t_n \hs\forall u \leq s_n)\\
&\leq \P\left(\exists w\in N(s_n) : X_w(u)>\sqrt2 u - a s_n^{1/3} + \frac{\eps}{2}\log s_n \hs\forall u \leq s_n\right).
\end{align*}
Noting that $s_n\geq t_n/2$, by the upper bound in Theorem \ref{bram} the above is at most
\[\frac{c_2\eps}{2} s_n^{-\eps/\sqrt2}\log s_n \leq c_2\eps t_n^{-\eps/\sqrt2} \log t_n.\]
Thus, since $e^{2t_{k-1}} = \delta\log t_k$,
\begin{align*}
\P\bigg(E_k\cap F_k \bigg| \bigcap_{j=n}^{k-1}(E_j\cap F_j)\bigg) &\leq \P\bigg(E_k \bigg| \bigcap_{j=n}^{k-1}(E_j\cap F_j)\bigg)\\
&\leq \P\bigg(\bigcup_{v\in N(t_{k-1})} E^v_k \bigg| \bigcap_{j=n}^{k-1}(E_j\cap F_j)\bigg)\\
&\leq e^{2t_{k-1}} t_k^{-\eps/\sqrt2} e^{O(\log\log t_k)}\\
&\leq t_k^{-\eps/\sqrt2 + \sqrt2\delta} e^{O(\log\log t_k)}.
\end{align*}
Since we chose $\delta<\varepsilon/2$, this tends to zero as $k\to\infty$.
\end{proof}

\begin{lem}\label{exist_ev}
\[\limsup_{t\to\infty} \frac{\Lambda(t) - a t^{1/3}}{\log t} \leq 0 \hs \hbox{ almost surely}.\]
\end{lem}

\begin{proof}
We show that for large $t$ and any $\varepsilon>0$, there are always particles that have stayed above $\sqrt{2}u - at^{1/3} - \eps\log t$ for all times $u\leq t$.
Let
\[A_t = \{\not\exists v\in N(t) : X_v(u) > \sqrt2 u - at^{1/3} \hs \forall u\leq t\}\]
and
\[B_t = \{|N(\varepsilon\log t)|\geq t^{\varepsilon/2}, |X_v(\eps\log t)| \leq \sqrt2\varepsilon \log t \hs\forall v\in N(\varepsilon\log t)\}.\]
As before we write $N^v(t)$ for the set of descendants of particle $v$ that are alive at time $t$. Let $l_t = t - \varepsilon\log t$. Then for all large $t$,
\begin{align*}
&\P(A_t\cap B_t)\\
&\leq \E\left[\prod_{v\in N(\varepsilon\log t)} \hspace{-4mm}\P(\not\exists w\in N^v(t) : X_w(t)>\sqrt2 u - at^{1/3} \hs \forall u\leq t |\F_{\eps\log t}) \ind_{B_t}\right]\\
&\leq \E\left[\prod_{v\in N(\log t)}\P(\not\exists w\in N(l_t) : X_w(u)>\sqrt2 u - a t^{1/3} \hs \forall u\leq l_t)\ind_{B_t}\right]\\
&\leq \E\left(\not\exists w\in N(l_t) : X_w(u)>\sqrt2 u - a l_t^{1/3} \hs \forall u\leq l_t\right)^{t^{\eps/2}}.
\end{align*}
By the lower bound in Theorem \ref{bram} there exists $c>0$ such that this is at most
\[(1-c)^{t^{\eps/2}}.\]
Thus by Borel-Cantelli, for any lattice times $t_n\to\infty$, $\P(A_{t_n}\cap B_{t_n} \hbox{ infinitely often})=0$. But for all large $t$, $|N(\varepsilon\log t)|\geq t^{\varepsilon/2}$ and $|X_v(\varepsilon\log t)| \leq \sqrt2 \varepsilon\log t$ for all $v\in N(\eps\log t)$, so we deduce that $\P(A_{t_n} \hbox{ infinitely often}) = 0$. Then if we choose $t_n - t_{n-1}$ small enough, $-a t^{1/3} -\eps\log t < -a t_n^{1/3}$ for all $t\in(t_{n-1},t_n)$, so the result holds.
\end{proof}

\subsection{A lower bound on the $\liminf$}
\begin{lem}\label{noone_ev}
\[\liminf_{t\to\infty} \frac{\Lambda(t) - a t^{1/3}}{\log t} \geq -1/3\sqrt2 \hs \hbox{ almost surely}.\]
\end{lem}

\begin{proof}
Fix $\eps>0$. We show that for large $t$, there are no particles that stay above $\sqrt{2}u - at^{1/3}+((1+\eps)/3\sqrt2)\log t$ for all times $u\leq t$.

Choose $M$ large enough that for any $n\geq M$ and any $t\in [n,2n]$, we have
\[-at^{1/3} + \frac{(1+\eps)}{3\sqrt2}\log t \geq -as^{1/3} + \frac{(1+\eps)}{3\sqrt2}\log s -1 \hs\hs \forall s\in [t-n^{2/3},t].\]
Define
\[U_t = \{\exists v\in N(t) : X_v(u) > \sqrt2 u - at^{1/3} + \smfr{(1+\eps)}{3\sqrt2}\log t \hs \forall u\leq t\}\]
and
\[U'_t = \{\exists v\in N(t) : X_v(u) > \sqrt2 u - at^{1/3} + \smfr{(1+\eps)}{3\sqrt2}\log t - 1 \hs \forall u\leq t\}.\]
Let
\[ I_n = \int_n^{2n}\ind_{U_t} dt \hs\hs\hs\hbox{ and } \hs\hs\hs I'_n = \int_{n-n^{2/3}}^{2n} \ind_{U'_t} dt.\]
Note that
\begin{equation}\label{1mtrick}
\P(I_n>0) \leq \frac{\E[I'_n]}{\E[I'_n\ind_{\{I_n>0\}}]}\P(I_n>0) = \frac{\E[I'_n]}{\E[I'_n|I_n>0]}.
\end{equation}

If $I_n>0$, then there exists $t\in[n,2n]$ and $v\in N(t)$ such that $X_v(u)>\sqrt2 u - at^{1/3} + \frac{(1+\eps)}{3\sqrt2}\log t$ for all $u\leq t$. But then if $n\geq M$, $U'_s$ occurs for all $s\in [t-n^{2/3},t]$. We deduce that for $n\geq M$,
\[\E[I'_n | I_n>0] \geq n^{2/3}.\]
But by Theorem \ref{bram},
\begin{align*}
\E[I'_n] &= \int_{n-n^{2/3}}^{2n} \P(U'_t) dt\\
&\leq \int_{n-n^{2/3}}^{2n} c_2 \left(\frac{(1+\eps)}{3\sqrt2}\log t - 1\right) e^{-\frac{(1+\eps)}{3}\log t + \sqrt2} dt\\
&\lesssim n^{(2-\eps)/3}\log n.
\end{align*}
Plugging these estimates back into (\ref{1mtrick}), we get
\[\P(I_n>0) \lesssim n^{-\eps/3}\log n.\]
By Borel-Cantelli, the probability that there exist infinitely many $k$ with $I_{2^k}>0$ is zero. Since $\bigcup_{k\geq1} [2^k, 2\cdot 2^k) = [2,\infty)$ and $\eps>0$ was arbitrary, we deduce the result.
\end{proof}

\subsection{An upper bound on the $\liminf$}

The final lemma in our series requires the most work.

\begin{lem}\label{exist_io}
\[\liminf_{t\to\infty} \frac{\Lambda(t) - a t^{1/3}}{\log t} \leq -1/3\sqrt2 \hs \hbox{ almost surely}.\]
\end{lem}

To prove this lemma, let
\[f_t(u) = \sqrt2 u - at^{1/3} + \frac{1}{3\sqrt2}\log t + 1\]
and
\[L_t(u) = a(t-u)^{1/3}.\]
Define
\[\mathcal A_t(u) = \{v\in N(u) : X_v(r) - f_t(r)\in (0,L_t(r))\hsl\forall r\leq u, \hs X_v(u)-f_t(u)\in [L_t(u)-2,L_t(u)-1)\},\]
the set of particles that have stayed in the $(f_t,f_t+L_t)$-tube up to time $u$ and are near the top of the tube at time $u$. Also let
\[A_t(u) = \{\xi_r - f_t(r)\in (0,L_t(r))\hsl\forall r\leq u, \hs \xi_u-f_t(u)\in [L_t(u)-2,L_t(u)-1)\},\]
the event that our single Brownian motion is ``in $\mathcal A_t(u)$''.

Set $D_t = [t/3,2t/3]\cap \mathbb{N}$ and $K_n = \{n+kn^{2/3} : k=1,\ldots,\lfloor n^{1/3}\rfloor\}$. Define
\[\mathcal M_t = \sum_{j\in D_t} \#\mathcal A_t(j)\]
and
\[S_n = \sum_{t\in K_n} \mathcal M_t.\]
Our aim is to show that $\P(S_n>0)$ is bounded away from $0$. We do this by estimating moments of $S_n$. We will use notation from Section \ref{bram_proof}.

\begin{lem}\label{I1m}
For all $n\geq 2$,
\[\E[S_n] \asymp \log n.\]
\end{lem}

\begin{proof}
Note that
\[\E[S_n] = \sum_{t\in K_n} \sum_{j\in D_t} \E[\#\mathcal A_t(j)] = \sum_{t\in K_n} \sum_{j\in D_t} e^j\P( A_t(j) ).\]
But by Lemma \ref{1pb}, for $j\in D_t$,
\[\P( A_t(j) ) = p(t,\frac{1}{3\sqrt2}\log t;j,1) \asymp q(t,\frac{1}{3\sqrt2}\log t;j,1) = \left(\frac{1}{3\sqrt2}\log t\right)e^{-j-\frac13\log t +\sqrt2}t^{-1/2}(t+1-j)^{-1/2}.\]
Thus
\[\E[S_n] \asymp \sum_{t\in K_n} \frac{\log n}{n^{1/3}} \asymp \log n.\qedhere\]
\end{proof}

\begin{lem}\label{I2m}
For all $n\geq 2$,
\[\E[S_n^2] \lesssim \log^2 n.\]
\end{lem}

\begin{proof}
First note that
\[\E[S_n^2] \leq 2 \sum_{t\in K_n} \sum_{\substack{s\in K_n\\s\leq t}} \sum_{i\in D_s}\sum_{j\in D_t} \E[\#\mathcal A_s(i)\cdot \#\mathcal A_t(j)].\]

Suppose first that $i\leq j$. Recall our two dependent Brownian motions $\xi^{(1)}$ and $\xi^{(2)}$ from the proof of Lemma \ref{jaff2mom}. Write $A^{(1)}_s(i)$ to mean the analogue of $A_s(i)$ but with $\xi^{(1)}$ in place of $\xi$, and similarly for $A^{(2)}_t(j)$. Then by the many-to-two lemma (specifically Example 6 of \cite{harris_roberts:many_to_few}),
\begin{equation*}
\E[\#\mathcal A_s(i)\cdot \#\mathcal A_t(j)] = e^j\P(A_s(i)\cap A_t(j)) + 2\int_0^i e^{i+j-r}\P(A_s^{(1)}(i)\cap A_t^{(2)}(j) \hsl |\hsl T=r) dr.
\end{equation*}
Let
\begin{align*}
\Delta_{s,t}(u) &= f_t(u)+L_t(u) - (f_s(u)+L_s(u))\\
&= -at^{1/3} + \frac{1}{3\sqrt2}\log t + a(t-u)^{1/3} + as^{1/3} - \frac{1}{3\sqrt2}\log s - a(s-u)^{1/3},
\end{align*}
the distance between the tops of the $t$- and $s$-tubes at time $u$.

We begin by estimating $\P(A_s(i)\cap A_t(j))$ by applying Lemma \ref{1pb} appropriately. Note that
\[\P(A_s(i)\cap A_t(j)) \lesssim p(s,\smfr{1}{3\sqrt2}\log s; i,1) p(t-i,\Delta_{s,t}(i);j-i,1).\]
Thus by Lemma \ref{1pb}, if $j-i\leq (t-i)^{2/3}$ then
\begin{align*}
\P(A_s(i)\cap A_t(j)) &\lesssim q(s,\smfr{1}{3\sqrt2}\log s; i,1) \cdot \tilde q(\Delta_{s,t}(i);j-i,1)\\
&\lesssim (\log s) e^{-i-\frac13\log s + \sqrt2}s^{-1/2}(s+1-i)^{-1/2}\\
&\hspace{40mm} \cdot \Delta_{s,t}(i) e^{-j+i -\sqrt2\Delta_{s,t}(i)+\sqrt2}(j-i+1)^{-3/2}.
\end{align*}
On the other hand, if $j-i>(t-i)^{2/3}$ then
\begin{align*}
\P(A_s(i)\cap A_t(j)) &\lesssim q(s,\smfr{1}{3\sqrt2}\log s; i,1) \cdot q(t-i,\Delta_{s,t}(i);j-i,1)\\
&= (\smfr{1}{3\sqrt2}\log s) e^{-i-\frac13\log s + \sqrt2}s^{-1/2}(s+1-i)^{-1/2}\\
&\hspace{40mm} \cdot \Delta_{s,t}(i) e^{-j+i -\sqrt2\Delta_{s,t}(i)+\sqrt2}(t-i)^{-1/2}(t+1-j)^{-1/2}.
\end{align*}
Thus for any $s,t\in K_n$ and $i\in D_s$ we have
\[\sum_{\substack{j\in D_t\\j\geq i}} e^j \P(A_s(i)\cap A_t(j)) \lesssim \Delta_{s,t}(i)e^{-\sqrt2 \Delta_{s,t}(i)} n^{-4/3}\log n\]
and hence (since $i\geq s/3$ and $\Delta_{s,t}(u)$ is increasing in $u$)
\[\sum_{i\in D_s} \sum_{\substack{j\in D_t\\j\geq i}} e^j \P(A_s(i)\cap A_t(j)) \lesssim \Delta_{s,t}(s/3)e^{-\sqrt2 \Delta_{s,t}(s/3)} n^{-1/3}\log n.\]
If $\Delta_{s,t}(s/3) \geq \smfr{1}{3\sqrt2}\log n$, then this is at most $n^{-2/3}\log^2 n$. Also, for each $t$, there are $O(\log n)$ values of $s$ within $K_n$ such that $\Delta_{s,t}(s/3) < \smfr{1}{3\sqrt2}\log n$, and in this case the above is at most $n^{-1/3}\log n$. Thus
\[\sum_{t\in K_n} \sum_{\substack{s\in K_n\\s\leq t}} \sum_{i\in D_s}\sum_{\substack{j\in D_t\\ j\geq i}} e^j \P(A_s(i)\cap A_t(j)) \lesssim \log^2 n.\]

We now set
\[J_{s,t}(i,j,r):=\P(A_s^{(1)}(i)\cap A_t^{(2)}(j) \hsl |\hsl T=r)\]
and proceed to estimating $\int_0^i e^{i+j-r} J_{s,t}(i,j,r) dr$ by integrating out the value of $\xi^{(1)}_r$ and considering several cases depending on the value of $r$. We will assume throughout that $r\leq i$, $i\in D_s$ and $j\in D_t$, since these are the values in which we are interested. Note that
\[J_{s,t}(i,j,r) \lesssim \int_0^{L_s(r)} p\left(s,\hbox{$\frac{1}{3\sqrt2}$}\log s; r,y\right)p(s-r,y;i-r,1)p(t-r, y+\Delta_{s,t}(r);j-r,1) dy.\]
The calculations that follow are mechanical and repetitive, but necessary.

\vspace{3mm}

\noindent
First suppose that $r\leq s^{2/3}$. Then
\begin{align*}
J_{s,t}(i,j,r) &\lesssim \int_0^{L_s(r)} \tilde q\left(\hbox{$\frac{1}{3\sqrt2}$}\log s; r,y\right)q(s-r,y;i-r,2)q(t-r, y+\Delta_{s,t}(r);j-r,2) dy\\
&\leq \int_0^{L_s(r)} y(\smfr{1}{3\sqrt2}\log s)e^{-r-\frac13\log s + \sqrt2 y} (r+1)^{-3/2}\\
&\hspace{20mm} \cdot 2y e^{-(i-r)-\sqrt2 y + 2\sqrt2} (s-r)^{-1/2}(s+1-i)^{-1/2}\\
&\hspace{27mm} \cdot 2(y+\Delta_{s,t}(r))e^{-(j-r)-\sqrt2(y+\Delta_{s,t}(r))+2\sqrt2} (t-r)^{-1/2}(t+1-j)^{-1/2} dy\\
&\asymp (\Delta_{s,t}(r)+1)e^{r-i-j-\sqrt2\Delta_{s,t}(r)}(r+1)^{-3/2} n^{-7/3}\log n.
\end{align*}
Thus in this case (using the fact that $\Delta_{s,t}(u)$ is increasing in $u$)
\[\int_0^i e^{i+j-r} J_{s,t}(i,j,r) dr \lesssim  (\Delta_{s,t}(0)+1)e^{-\sqrt2\Delta_{s,t}(0)} n^{-7/3} \log n.\]

\vspace{3mm}

\noindent
Secondly suppose that $r\geq s^{2/3}$, $i-r\geq (s-r)^{2/3}$ and $j-r \geq (t-r)^{2/3}$. Then
\begin{align*}
J_{s,t}(i,j,r) &\lesssim \int_0^{L_s(r)} q\left(s,\hbox{$\frac{1}{3\sqrt2}$}\log s; r,y+1\right)q(s-r,y;i-r,2)q(t-r, y+\Delta_{s,t}(r);j-r,2) dy\\
&\leq \int_0^{L_s(r)} (y+1)(\smfr{1}{3\sqrt2}\log s)e^{-r-\frac13\log s + \sqrt2 y+\sqrt2} s^{-1/2}(s+1-r)^{-1/2}\\
&\hspace{20mm} \cdot 2y e^{-(i-r)-\sqrt2 y + 2\sqrt2} (s-r)^{-1/2}(s+1-i)^{-1/2}\\
&\hspace{27mm} \cdot 2(y+\Delta_{s,t}(r))e^{-(j-r)-\sqrt2(y+\Delta_{s,t}(r))+2\sqrt2} (t-r)^{-1/2}(t+1-j)^{-1/2} dy\\
&\asymp (\Delta_{s,t}(r)+1)e^{r-i-j-\sqrt2\Delta_{s,t}(r)} n^{-10/3}\log n.
\end{align*}
Thus in this case
\[\int_0^i e^{i+j-r} J_{s,t}(i,j,r) dr \lesssim (\Delta_{s,t}(0)+1)e^{-\sqrt2\Delta_{s,t}(0)} n^{-7/3} \log n.\]

\vspace{3mm}

\noindent
Thirdly suppose that $i-r\leq (s-r)^{2/3}$ and $j-r\geq (t-r)^{2/3}$. Then
\begin{align*}
J_{s,t}(i,j,r) &\lesssim \int_0^{L_s(r)} q\left(s,\hbox{$\frac{1}{3\sqrt2}$}\log s; r,y+1\right)\tilde q(y;i-r,1)q(t-r, y+\Delta_{s,t}(r);j-r,2) dy\\
&\leq \int_0^{L_s(r)} (y+1)(\smfr{1}{3\sqrt2}\log s)e^{-r-\frac13\log s + \sqrt2 y+\sqrt2} s^{-1/2}(s+1-r)^{-1/2}\\
&\hspace{20mm} \cdot (y(i-r)^{-3/2}\wedge 1) e^{-(i-r)-\sqrt2 y + \sqrt2}\\
&\hspace{27mm} \cdot 2(y+\Delta_{s,t}(r))e^{-(j-r)-\sqrt2(y+\Delta_{s,t}(r))+2\sqrt2} (t-r)^{-1/2}(t+1-j)^{-1/2} dy\\
&\asymp (\Delta_{s,t}(r)+1)e^{r-i-j-\sqrt2\Delta_{s,t}(r)} (i+1-r)^{-3/2} n^{-7/3}\log n.
\end{align*}
Thus in this case
\[\int_0^i e^{i+j-r} J_{s,t}(i,j,r) dr \lesssim (\Delta_{s,t}(0)+1)e^{-\sqrt2\Delta_{s,t}(0)} n^{-7/3} \log n.\]

\vspace{3mm}

\noindent
Fourthly suppose that $i-r\geq (s-r)^{2/3}$ and $j-r\leq (t-r)^{2/3}$. Then
\begin{align*}
J_{s,t}(i,j,r) &\lesssim \int_0^{L_s(r)} q\left(s,\hbox{$\frac{1}{3\sqrt2}$}\log s; r,y+1\right)q(s-r,y;i-r,2)\tilde q(t-r, y+\Delta_{s,t}(r);j-r,1) dy\\
&\leq \int_0^{L_s(r)} (y+1)(\smfr{1}{3\sqrt2}\log s)e^{-r-\frac13\log s + \sqrt2 y+\sqrt2} s^{-1/2}(s+1-r)^{-1/2}\\
&\hspace{20mm} \cdot 2y e^{-(i-r)-\sqrt2 y + 2\sqrt2} (s-r)^{-1/2}(s+1-i)^{-1/2}\\
&\hspace{27mm} \cdot ((y+\Delta_{s,t}(r))(j-r)^{-3/2}\wedge 1)e^{-(j-r)-\sqrt2(y+\Delta_{s,t}(r))+\sqrt2} dy\\
&\asymp (\Delta_{s,t}(r)+1)e^{r-i-j-\sqrt2\Delta_{s,t}(r)} (j+1-r)^{-3/2} n^{-7/3}\log n.
\end{align*}
Thus in this case
\[\int_0^i e^{i+j-r} J_{s,t}(i,j,r) dr \lesssim (\Delta_{s,t}(0)+1)e^{-\sqrt2\Delta_{s,t}(0)} n^{-7/3} \log n.\]
(We could include an extra factor of $(j+1-i)^{-3/2}$ on the right-hand side above but we will not need it.)

\vspace{3mm}

\noindent
Fifthly, and finally, suppose that $i-r\leq (s-r)^{2/3}$ and $j-r\leq (t-r)^{2/3}$. Then
\begin{align*}
J_{s,t}(i,j,r) &\lesssim \int_0^{L_s(r)} q\left(s,\hbox{$\frac{1}{3\sqrt2}$}\log s; r,y+1\right)\tilde q(y;i-r,1)\tilde q(t-r, y+\Delta_{s,t}(r);j-r,1) dy\\
&\leq \int_0^{L_s(r)} (y+1)(\smfr{1}{3\sqrt2}\log s)e^{-r-\frac13\log s + \sqrt2 y+\sqrt2} s^{-1/2}(s+1-r)^{-1/2}\\
&\hspace{20mm} \cdot (y(i-r)^{-3/2}\wedge 1) e^{-(i-r)-\sqrt2 y + \sqrt2}\\
&\hspace{27mm} \cdot ((y+\Delta_{s,t}(r))(j-r)^{-3/2}\wedge 1)e^{-(j-r)-\sqrt2(y+\Delta_{s,t}(r))+\sqrt2} dy\\
&\asymp (\Delta_{s,t}(r)+1)e^{r-i-j-\sqrt2\Delta_{s,t}(r)} (i+1-r)^{-3/2} (j+1-r)^{-3/2} n^{-4/3}\log n.
\end{align*}
Thus in this case
\[\int_0^i e^{i+j-r} J_{s,t}(i,j,r) dr \lesssim (\Delta_{s,t}(0)+1)e^{-\sqrt2\Delta_{s,t}(0)} (j+1-i)^{-3/2} n^{-4/3} \log n.\]

\vspace{3mm}

In any of the five cases above, we have
\[\sum_{i\in D_s}\sum_{\substack{j\in D_t\\ j\geq i}} \int_0^i e^{i+j-r} J_{s,t}(i,j,r) dr \lesssim (\Delta_{s,t}(0)+1)e^{-\sqrt2\Delta_{s,t}(0)} n^{-1/3} \log n.\]
Now, if $\Delta_{s,t}(0)\geq \frac{1}{3\sqrt2}\log n$ then the above is at most $n^{-2/3}\log^2 n$, and if $\Delta_{s,t}(0) < \frac{1}{3\sqrt2}\log n$ then the above is at most $n^{-1/3}\log n$, but for any $t\in K_n$ there are only $O(\log n)$ values of $s$ within $K_n$ such that $\Delta_{s,t}(0)<\frac{1}{3\sqrt2}\log n$. We deduce that 
\[\sum_{t\in K_n} \sum_{\substack{s\in K_n\\s\leq t}} \sum_{i\in D_s}\sum_{\substack{j\in D_t\\ j\geq i}} \int_0^i e^{i+j-r} J_{s,t}(i,j,r) dr \lesssim \log^2 n.\]
Very similar calculations in the case $j<i$ give the same result, which completes the proof.
\end{proof}

\begin{proof}[Proof of Lemma \ref{exist_io}]
By Cauchy-Schwarz,
\[\P(S_n > 0) \geq \frac{\E[S_n]^2}{\E[S_n^2]}\gtrsim 1.\]
As in the proof of Theorem \ref{bram}, we can apply the Markov property and Proposition \ref{ppp} to see that
\[\P\left(\exists t\in [n,2n], v\in N(t) : X_v(u) \geq \sqrt2 u - at^{1/3} + \frac{1}{3\sqrt2}\log t \hsl \forall u\leq t\right)\geq \eta\]
for some $\eta>0$. Now an argument almost identical to that in Lemma \ref{exist_ev} completes the proof. We look at time $\eps\log t$ with $\eps>0$, and check that there are at least $n^{\eps/2}$ particles within distance $2\eps \log t$ of the origin. Then we apply the estimate above to see that the probability that none of these particles has a descendant which stays above $\sqrt2 u - at^{1/3} + (\frac{1}{\sqrt2}-2\eps)\log t$ for all $u\leq t$, for some $t\in [n,2n]$, is at most a constant times
\[(1-\eta)^{n^{\eps/2}}.\]
Applying the Borel-Cantelli lemma we see that, infinitely often as $t\to\infty$, there are particles that stay above $\sqrt2 u - at^{1/3} + (\frac{1}{\sqrt2}-2\eps)\log t$ for all $u\leq t$. Since $\eps>0$ was arbitrary, we obtain the desired result.
\end{proof}

The proof of Theorem \ref{hush} now follows by combining Lemmas \ref{noone_io}, \ref{noone_ev}, \ref{exist_ev} and \ref{exist_io}.

\subsection*{Acknowledgements}
Thanks to Louigi Addario-Berry for suggesting that I work on this problem, and for several helpful conversations. Thanks also to Gabriel Faraud for introducing me to \cite{faraud_hu_shi:conv_biased_rws_trees}, and to an anonymous referee for many helpful suggestions. This project was partially supported by a CRM-ISM postdoctoral fellowship at McGill University, partially by the Department of Statistics at the University of Warwick, and partially by EPSRC grant EP/K007440/1.

\bibliographystyle{plain}
\def\cprime{$'$}

\end{document}